\documentclass[reqno]{amsart}

\usepackage{amsmath,amsfonts,amssymb,amsthm}

\newtheorem{thm}{Theorem}
\newtheorem{lem}{Lemma}

\begin{document}

\title{On divisors of Lucas and Lehmer numbers}
\author{C.L.~Stewart}
\address{C.L.~Stewart \\ Department of Pure Mathematics \\ University of Waterloo \\ Waterloo, Ontario \\ Canada \\ email: cstewart@uwaterloo.ca}

\begin{abstract}
Let $u_n$ be the n-th term of a Lucas sequence or a Lehmer sequence.  In this article we shall establish an estimate from below for the greatest prime factor of $u_n$ which is of the form $n \exp (\log n/104 \log\log n)$.  In so doing we are able to resolve a question of Schinzel from 1962 and a conjecture of Erd\H{o}s from 1965.  In addition we are able to give the first general improvement on results of Bang from 1886 and Carmichael from 1912.
\end{abstract}


\keywords{greatest prime factor, recurrence sequences, Lucas numbers. \\ Research supported in part by the Canada Research Chairs Program and by Grant A3528 from the Natural Sciences and Engineering Research Council of Canada.}

\maketitle

\section{Introduction}

Let $\alpha$ and $\beta$ be complex numbers such that $\alpha+\beta$ and $\alpha\beta$ are non-zero coprime integers and $\alpha/\beta$ is not a root of unity. Put
$$
u_n=(\alpha^n-\beta^n)/(\alpha-\beta)\quad \text{for}\ n\geq 0.
$$
The integers $u_n$ are known as Lucas numbers and their divisibility properties have been studied by Euler, Lagrange, Gauss, Dirichlet and others (see \cite[Chapter XVII]{Di}). In 1876 Lucas \cite{Lu1} announced several new results concerning Lucas sequences $(u_n)^\infty_{n=0}$ and in a substantial paper in 1878 \cite{Lu2} he gave a systematic treatment of the divisibility properties of Lucas numbers and indicated some of the contexts in which they appeared. Much later Matijasevic \cite{M} appealed to these properties in his solution of Hilbert's 10th problem. 

For any integer $m$ let $P(m)$ denote the greatest prime factor of $m$ with the convention that $P(m)=1$ when $m$ is 1, 0 or $-1.$ In 1912 Carmichael \cite{Ca} proved that if $\alpha$ and $\beta$ are real and $n>12$ then
\begin{equation} \label{eq1}
P(u_n)\geq n-1.
\end{equation}

Results of this character had been established earlier for integers of the form $a^n-b^n$ where $a$ and $b$ are integers with $a>b>0.$ Indeed Zsigmondy \cite{Z} in 1892 and Birkhoff and Vandiver \cite{BV} in 1904 proved that for $n>2$
\begin{equation} \label{eq2}
P(a^n-b^n)\geq n+1,
\end{equation}
while in the special case that $b=1$ the result is due to Bang \cite{Ba} in 1886.

In 1930 Lehmer \cite{Le} showed that the divisibility properties of Lucas numbers hold in a more general setting. Suppose that $(\alpha+\beta)^2$ and $\alpha\beta$ are coprime non-zero integers with $\alpha/\beta$ not a root of unity and, for $n>0,$ put
$$
\tilde{u}_n=\begin{cases} (\alpha^n-\beta^n)/(\alpha-\beta) & \text{for $n$ odd,} \\
(\alpha^n-\beta^n)/(\alpha^2-\beta^2) & \text{for $n$ even.}
\end{cases}
$$
Integers of the above form have come to be known as Lehmer numbers. Observe that Lucas numbers are also Lehmer numbers up to a multiplicative factor of $\alpha+\beta$ when $n$ is even. In 1955 Ward \cite{W} proved that if $\alpha$ and $\beta$ are real then for $n>18,$
\begin{equation} \label{eq3}
P(\tilde{u}_n)\geq n-1,
\end{equation}
and four years later Durst \cite{Du} observed that \eqref{eq3} holds for $n>12.$

A prime number $p$ is said to be a primitive divisor of a Lucas number $u_n$ if $p$ divides $u_n$ but does not divide $(\alpha-\beta)^2u_2\cdots u_{n-1}.$ Similarly $p$ is said to be a primitive divisor of a Lehmer number $\tilde{u}_n$ if $p$ divides $\tilde{u}_n$ but does not divide $(\alpha^2-\beta^2)^2\tilde{u}_3\cdots\tilde{u}_{n-1}.$ For any integer $n>0$ and any pair of complex numbers $\alpha$ and $\beta,$ we denote the $n$-th cyclotomic polynomial in $\alpha$ and $\beta$ by $\Phi_n(\alpha,\beta),$ so
$$
\Phi_n(\alpha,\beta)=\prod^n_{\substack{j=1 \\ (j,n)=1}}(\alpha-\zeta^j\beta),
$$
where $\zeta$ is a primitive $n$-th root of unity. One may check, see \cite{St2}, that $\Phi_n(\alpha,\beta)$ is an integer for $n>2$ if $(\alpha+\beta)^2$ and $\alpha\beta$ are integers. Further, see Lemma 6 of \cite{St2}, if, in addition, $(\alpha+\beta)^2$ and $\alpha\beta$ are coprime non-zero integers, $\alpha/\beta$ is not a root of unity and $n>4$ and $n$ is not 6 or 12 then $P(n/(3,n))$ divides $\Phi_n(\alpha,\beta)$ to at most the first power and all other prime factors of $\Phi_n(\alpha,\beta)$ are congruent to 1 or $-1$ modulo $n.$ The last assertion can be strengthened to all other prime factors of $\Phi_n(\alpha,\beta)$ are congruent to $1\pmod{n}$ in the case that $\alpha$ and $\beta$ are coprime integers.

Since
\begin{equation} \label{eq4}
\alpha^n-\beta^n=\prod_{d\mid n}\Phi_d(\alpha,\beta),
\end{equation}
$\Phi_1(\alpha,\beta)=\alpha-\beta$ and $\Phi_2(\alpha,\beta)=\alpha+\beta$ we see that if $n$ exceeds 2 and $p$ is a primitive divisor of a Lucas number $u_n$ or Lehmer number $\tilde{u}_n$ then $p$ divides $\Phi_n(\alpha,\beta).$  Further, a primitive divisor of a Lucas number $u_n$ or Lehmer number $\tilde{u}_n$ is not a divisor of $n$ and so it is congruent to $\pm 1\pmod{n}.$ Estimates \eqref{eq1}, \eqref{eq2} and \eqref{eq3} follow as consequences of the fact that the $n$-th term of the sequences in question possesses a primitive divisor. It was not until 1962 that this approach was extended to the case where $\alpha$ and $\beta$ are not real by Schinzel \cite{Sc2}. He proved, by means of an estimate for linear forms in two logarithms of algebraic numbers due to Gelfond \cite{G}, that there is a positive number $C,$ which is effectively computable in terms of $\alpha$ and $\beta,$ such that if $n$ exceeds $C$ then $\tilde{u}_n$ possesses a primitive divisor. In 1974 Schinzel \cite{Sc6} employed an estimate of Baker \cite{B1} for linear forms in the logarithms of algebraic numbers to show that $C$ can be replaced by a positive number $C_0,$ which does not depend on $\alpha$ and $\beta,$ and in 1977 Stewart \cite{St3} showed $C_0$ could be taken to be $e^{452}4^{67}.$ This was subsequently refined by Voutier \cite{Va1,Va2} to 30030. In addition Stewart \cite{St3} proved that $C_0$ can be taken to be 6 for Lucas numbers and 12 for Lehmer numbers with finitely many exceptions and that the exceptions could be determined by solving a finite number of Thue equations. This program was successfully carried out by Bilu, Hanrot and Voutier \cite{BHV} and as a consequence they were able to show that for $n>30$ the $n$-th term of a Lucas or Lehmer sequence has a primitive divisor. Thus \eqref{eq1} and \eqref{eq3} hold for $n>30$ without the restriction that $\alpha$ and $\beta$ be real.

In 1962 Schinzel \cite{Sc1} asked if there exists a pair of integers $a,b$ with $ab$ different from $\pm 2c^2$ and $\pm c^h$ with $h\geq 2$ for which $P(a^n-b^n)$ exceeds $2n$ for all sufficiently large $n.$ In 1965 Erd\H{o}s \cite{E} conjectured that
$$
\frac{P(2^n-1)}{n}\rightarrow\infty\quad \text{as}\ n\rightarrow\infty.
$$
Thirty five years later Murty and Wong \cite{MW} showed that Erd\H{o}s' conjecture is a consequence of the $abc$ conjecture \cite{SY1}. They proved, subject to the $abc$ conjecture, that if $\varepsilon$ is a positive real number and $a$ and $b$ are integers with $a>b>0$ then
$$
P(a^n-b^n)>n^{2-\varepsilon},
$$
provided that $n$ is sufficiently large in terms of $a,$ $b$ and $\varepsilon.$ In 2004 Murata and Pomerance \cite{MP} proved, subject to the Generalized Riemann Hypothesis, that 
\begin{equation} \label{eq5}
P(2^n-1)>n^{4/3}/\log\log n
\end{equation}
for a set of positive integers $n$ of asymptotic density 1.

The first unconditional refinement of \eqref{eq2} was obtained by Schinzel \cite{Sc1} in 1962. He proved that if $a$ and $b$ are coprime and $ab$ is a square or twice a square then
$$
P(a^n-b^n)\geq 2n+1
$$
provided that one excludes the cases $n=4,6,12$ when $a=2$ and $b=1.$ Schinzel proved his result by showing that the term $a^n-b^n$ was divisible by at least 2 primitive divisors. To prove this result he appealed to an Aurifeuillian factorization of $\Phi_n$. Rotkiewicz \cite{R} extended Schinzel's argument to treat Lucas numbers and then Schinzel \cite{Sc3,Sc4,Sc5} in a sequence of articles gave conditions under which Lehmer numbers possess at least 2 primitive divisors and so under which \eqref{eq3} holds with $n+1$ in place of $n-1,$ see also \cite{J}. In 1975 Stewart \cite{St1} proved that if $\kappa$ is a positive real number with $\kappa<1/\log 2$ then $P(a^n-b^n)/n$ tends to infinity with $n$ provided that $n$ runs through those integers with at most $\kappa\log\log n$ distinct prime factors, see also \cite{ES}. Stewart \cite{St2} in the case that $\alpha$ and $\beta$ are real and Shorey and Stewart \cite{SS} in the case that $\alpha$ and $\beta$ are not real generalized this work to Lucas and Lehmer sequences. Let $\alpha$ and $\beta$ be complex numbers such that $(\alpha+\beta)^2$ and $\alpha\beta$ are non-zero relatively prime integers with $\alpha/\beta$ not a root of unity. For any positive integer $n$ let $\omega(n)$ denote the number of distinct prime factors of $n$ and put $q(n)=2^{\omega(n)},$ the number of square-free divisors of $n.$ Further let $\varphi(n)$ be the number of positive integers less than or equal to $n$ and coprime with $n.$ They showed, recall \eqref{eq4}, if $n\ (>3)$ has at most $\kappa\log\log n$ distinct prime factors then
\begin{equation} \label{eq6}
P(\Phi_n(\alpha,\beta))>C(\varphi(n)\log n)/q(n),
\end{equation}
where $C$ is a positive number which is effectively computable in terms of $\alpha,$ $\beta$ and $\kappa$ only. The proofs depend on lower bounds for linear forms in the logarithms of algebraic numbers in the complex case when $\alpha$ and $\beta$ are real and in the $p$-adic case otherwise.

The purpose of the present paper is to answer in the affirmative the question posed by Schinzel \cite{Sc1} and to prove Erd\H{o}s' conjecture in the wider context of Lucas and Lehmer numbers.

\begin{thm} \label{thm1}
Let $\alpha$ and $\beta$ be complex numbers such that $(\alpha+\beta)^2$ and $\alpha\beta$ are non-zero integers and $\alpha/\beta$ is not a root of unity. There exists a positive number $C,$ which is effectively computable in terms of $\omega(\alpha\beta)$ and the discriminant of $\mathbb{Q}(\alpha/\beta),$ such that for $n>C,$
\begin{equation} \label{eq7}
P(\Phi_n(\alpha,\beta))>n\exp(\log n/104\log\log n).
\end{equation}
\end{thm}

Our result, with the aid of \eqref{eq4} gives an improvement of \eqref{eq1}, \eqref{eq2}, \eqref{eq3} and \eqref{eq6}, answers the question of Schinzel and proves the conjecture of Erd\H{o}s. Specifically, if $a$ and $b$ are integers with $a>b>0$ then 
\begin{equation} \label{eq8}P(a^n-b^n)>n\exp(\log n/104\log\log n),
\end{equation}
for $n$ sufficiently large in terms of the number of distinct prime factors of $ab.$ We remark that the factor 104 which occurs on the right hand side of \eqref{eq7} has no arithmetical significance. Instead it is determined by the current quality of the estimates for linear forms in $p$-adic logarithms of algebraic numbers. In fact we could replace 104 by any number strictly larger than $14e^2.$ The proof depends upon estimates for linear forms in the logarithms of algebraic numbers in the complex and the $p$-adic case. In particular it depends upon \cite{SY2} where improvements upon the dependence on the parameter $p$ in the lower bounds for linear forms in $p$-adic logarithms of algebraic numbers are established. This allows us to estimate directly the order of primes dividing $\Phi_n(\alpha,\beta).$ The estimates are non-trivial for small primes and, coupled with an estimate from below for $|\Phi_n(\alpha,\beta)|,$ they allow us to show that we must have a large prime divisor of $\Phi_n(\alpha,\beta)$ since otherwise the total non-archimedean contribution from the primes does not balance that of $|\Phi_n(\alpha,\beta)|.$ By contrast for the proof of \eqref{eq6} a much weaker assumption on the greatest prime factor is imposed and it leads to the conclusion that then $\Phi_n(\alpha,\beta)$ is divisible by many small primes. This part of the argument from \cite{SS} and \cite{St2} was also employed in Murata and Pomerance's \cite{MP} proof of \eqref{eq5} and in estimates of Stewart \cite{St4} for the greatest square-free factor of $\tilde{u}_n.$

For any non-zero integer $x$ let ord$_px$ denote the $p$-adic order of $x$.  Our next result follows from a special case of Lemma 8 of this paper.  Lemma 8 yields a crucial step in the proof of Theorem 1.  An unusual feature of the proof of Lemma 8 is that we artificially inflate the number of terms which occur in the $p$-adic linear form in logarithms which appears in the argument.  We have chosen to highlight it in the integer case.

\begin{thm} \label{thm2}
Let $a$ and $b$ be integers with $a>b>0.$ There exists a number $C_1,$ which is effectively computable in terms of $\omega(ab),$ such that if $p$ is a prime number which does not divide $ab$ and which exceeds $C_1$ and $n$ is an integer with $n\geq 2$ then
\begin{equation} \label{eq8} 
\text{\rm ord}_p(a^n-b^n)<p\exp(-\log p/52\log\log p)\log a + \text{\rm ord}_pn.
\end{equation}
\end{thm}

\vskip.1in

If $a$ and $b$ are integers with $a>b>0$, $p$ is an odd prime number which does not divide $a b$ and $n\geqslant 2$ then, as in the proof of Theorem 2, 
$$
\text{\rm ord}_p (a^n - b^n)\leq \text{\rm ord}_p (a^{p-1} - b^{p-1}) + \text{\rm ord}_p n. 
$$
In particular if $p$ exceeds $C_1$ then 

$$ \text{\rm ord}_p (a^{p-1} - b^{p-1}) < ~ p\exp(-\log p/52\log\log p) \log a.$$

\noindent
Yamada \cite{Ya}, by making use of a refinement of an estimate of Bugeaud and Laurent \cite{BL} for linear forms in two $p$-adic logarithms, proved that there is a positive number $C_2$, which is effectively computable in terms of $\omega(a)$, such that 

\begin{equation} \label{eq10} 
\text{\rm ord}_p(a^{p-1}-1) < C_2(p  / (\log p)^2) \log a. 
\end{equation}

\noindent
By following our proof of Theorem 1 and using (10) in place of Lemma 8 it is possible to show that there exist positive numbers $C_3,C_4$ and $C_5$, which are effectively computable in terms of $\omega(a)$ such that if $n$ exceeds $C_3$ then 

$$P(a^n - 1) > C_4 ~  \varphi (n) (\log n \log \log n)^{\frac{1}{2}}$$

\noindent
and so, by Theorem 328 of \cite{HW}, 

\begin{equation} \label{eq11}
P (a^n-1) > C_5 ~ n (\log n / \log \log n)^\frac{1}{2}.
\end{equation} 

\noindent
This gives an alternative proof of the conjecture of Erd\H{o}s, although the lower bound (11) is weaker than the bound (8).\\

The research for this paper was done in part during visits to the Hong Kong University of Science and Technology, Institut des Hautes \'Etudes Scientifiques and the Erwin Schr\"odinger International Institute for Mathematical Physics and I would like to express my gratitude to these institutions for their hospitality.  In addition I wish to thank Professor Kunrui Yu for helpful remarks concerning the presentation of this article and for our extensive discussions on estimates for linear forms in p-adic logarithms which led to \cite{SY2}.

\section{Preliminary lemmas}

Let $\alpha$ and $\beta$ be complex numbers such that $(\alpha+\beta)^2$ and $\alpha\beta$ are non-zero integers and $\alpha/\beta$ is not a root of unity. We shall assume, without loss of generality, that
$$
|\alpha|\geq|\beta|.
$$
Observe that
$$
\alpha=\frac{\sqrt{r}+\sqrt{s}}{2},\quad \beta=\frac{\sqrt{r}-\sqrt{s}}{2}
$$
where $r$ and $s$ are non-zero integers with $|r|\neq|s|.$ Further $\mathbb{Q}(\alpha/\beta)=\mathbb{Q}(\sqrt{rs}).$ Note that $(\alpha^2-\beta^2)^2=rs$ and we may write $rs$ in the form $m^2d$ with $m$ a positive integer and $d$ a square-free integer so that $\mathbb{Q}(\sqrt{rs})=\mathbb{Q}(\sqrt{d}).$

For any algebraic number $\gamma$ let $h(\gamma)$ denote the absolute logarithmic height of $\gamma.$ In particular if $a_0(x-\gamma_1)\cdots(x-\gamma_d)$ in $\mathbb{Z}[x]$ is the minimal polynomial of $\gamma$ over $\mathbb{Z}$ then
$$
h(\gamma)=\frac{1}{d}\left(\log a_0+\sum^d_{j=1}\log\max(1,|\gamma_j|)\right).
$$

Notice that 
$$
\alpha\beta(x-\alpha/\beta)(x-\beta/\alpha)=\alpha\beta x^2-(\alpha^2+\beta^2)x+\alpha\beta=\alpha\beta x^2-((\alpha+\beta)^2-2\alpha\beta)x+\alpha\beta
$$
is a polynomial with integer coefficients and so either $\alpha/\beta$ is rational or the polynomial is a multiple of the minimal polynomial of $\alpha/\beta.$ Therefore we have
\begin{equation} \label{eq9}
h(\alpha/\beta)\leq\log|\alpha|.
\end{equation}

We first record a result describing the prime factors of $\Phi_n(\alpha,\beta).$

\begin{lem} \label{lem1}
Suppose that $(\alpha+\beta)^2$ and $\alpha\beta$ are coprime. If $n>4$ and $n\neq 6,12$ then $P(n/(3n))$ divides $\Phi_n(\alpha,\beta)$ to at most the first power. All other prime factors of $\Phi_n(\alpha,\beta)$ are congruent to $\pm 1\pmod{n}.$
\end{lem}

\begin{proof}
This is Lemma 6 of \cite{St2}. 
\end{proof}

Let $K$ be a finite extension of $\mathbb{Q}$ and let $\wp$ be a prime ideal in the ring of algebraic integers $\mathcal{O}_K$ of $K.$ Let $\mathcal{O}_\wp$ consist of 0 and the non-zero elements $\alpha$ of $K$ for which $\wp$ has a non-negative exponent in the canonical decomposition of the fractional ideal generated by $\alpha$ into prime ideals. Then let $P$ be the unique prime ideal of $\mathcal{O}_\wp$ and put $\overline{K_\wp}=\mathcal{O}_\wp/P.$ Further for any $\alpha$ in $\mathcal{O}_\wp$ we let $\overline{\alpha}$ be the image of $\alpha$ under the residue class map that sends $\alpha$ to $\alpha+P$ in $\overline{K_\wp}.$

Our next result is motivated by work of Lucas \cite{Lu2} and Lehmer \cite{Le}.

\begin{lem} \label{lem2}
Let $d$ be a square-free integer different from $1,$ $\theta$ be an algebraic integer of degree $2$ over $\mathbb{Q}$ in $\mathbb{Q}(\sqrt{d})$ and let $\theta'$ denote the algebraic conjugate of $\theta$ over $\mathbb{Q}.$ Suppose that $p$ is a prime which does not divide $2\theta\theta'.$ Let $\wp$ be a prime ideal of the ring of algebraic integers of $\mathbb{Q}(\sqrt{d})$ lying above $p.$ The order of $\overline{\theta/\theta'}$ in $(\overline{\mathbb{Q}(\sqrt{d}})_\wp)^\times$ is a divisor of $2$ if $p$ divides $(\theta^2-\theta'^2)^2$ and a divisor of $p-(d/p)$ otherwise.
\end{lem}

\begin{proof}
We first note that $\theta$ and $\theta'$ are $p$-adic units. If $p$ divides $(\theta^2-\theta'^2)^2$ then either $p$ divides $(\theta-\theta')^2$ or $p$ divides $\theta+\theta'$ and in both cases $(\theta/\theta')^2\equiv 1\pmod{\wp}.$ Thus the order of $\overline{\theta/\theta'}$ divides 2.

Thus we may suppose that $p$ does not divide $2\theta\theta'(\theta^2-\theta'^2)^2$ and, in particular, $p\nmid d.$ Since
\begin{equation} \label{eq10}
2\theta=(\theta+\theta')+(\theta-\theta')\quad \text{and}\quad 2\theta'=(\theta+\theta')-(\theta-\theta')
\end{equation}
we see, on raising both sides of the above equations to the $p$-th power and subtracting, that $2^p(\theta^p-\theta'^p)-2(\theta-\theta')^p$ is $p(\theta-\theta')$ times an algebraic integer. Hence, since $p$ is odd,
$$
\frac{\theta^p-\theta'^p}{\theta-\theta'}\equiv (\theta-\theta')^{p-1}\pmod{p}.
$$
But
$$
(\theta-\theta')^{p-1}=((\theta-\theta')^2)^{\frac{p-1}{2}}\equiv\left(\frac{(\theta-\theta')^2}{p}\right)\pmod{p}
$$
and
$$
\left(\frac{(\theta-\theta')^2}{p}\right)=\left(\frac{d}{p}\right),
$$
so
\begin{equation} \label{eq11}
\frac{\theta^p-\theta'^p}{\theta-\theta'}\equiv\left(\frac{d}{p}\right)\pmod{p}.
\end{equation}

By raising both sides of equation \eqref{eq10} to the $p$-th power and adding we find that
$$
\frac{\theta^p+\theta'^p}{\theta+\theta'}\equiv(\theta+\theta')^{p-1}\pmod{p}
$$
and since $\left(\frac{(\theta+\theta')^2}{p}\right)=1,$
\begin{equation} \label{eq12}
\frac{\theta^p+\theta'^p}{\theta+\theta'}\equiv 1\pmod{p}.
\end{equation}
If $\left(\frac{d}{p}\right)=-1$ then adding \eqref{eq11} and \eqref{eq12} we find that
$$
2\frac{\theta^{p+1}-\theta'^{p+1}}{\theta^2-\theta'^2}\equiv 0\pmod{p}
$$
hence, since $p$ does not divide $2\theta\theta'(\theta^2-\theta'^2)^2,$
$$
(\theta/\theta')^{p+1}\equiv 1\pmod{\wp}
$$
and the result follows. If $\left(\frac{d}{p}\right)=1$ then subtracting \eqref{eq11} and \eqref{eq12} we find that
$$
2\theta\theta'\frac{\theta^{p-1}-\theta'^{p-1}}{\theta^2-\theta'^2}\equiv 0\pmod{p}
$$
hence, since $p$ does not divide $2\theta\theta'(\theta^2-\theta'^2)^2,$
$$
\left(\theta / \theta' \right)^{p-1}\equiv 1\pmod{\wp}
$$
and this completes the proof.
\end{proof}

Let $\ell$ and $n$ be integers with $n\geq 1$ and for each real number $x$ let $\pi(x,n,\ell)$ denote the number of primes not greater than $x$ and congruent to $\ell$ modulo $n.$ We require a version of the Brun-Titchmarsh theorem, see \cite[Theorem 3.8]{HR}.

\begin{lem} \label{lem3}
If $1\leq n< x$ and $(n,\ell)=1$ then
$$
\pi(x,n,\ell)<3x/(\varphi(n)\log(x/n)).
$$
\end{lem}

Our next result gives an estimate for the primes $p$ below a given bound which occur as the norm of an algebraic integer in the ring of algebraic integers of $\mathbb{Q}(\alpha/\beta).$ 

\begin{lem} \label{lem4}
Let $d$ be a squarefree integer with $d\neq 1$ and let $p_k$ denote the $k$-th smallest prime of the form $N\pi_k=p_k$ where $N$ denotes the norm from $\mathbb{Q}(\sqrt{d})$ to $\mathbb{Q}$ and $\pi_k$ is an algebraic integer in $\mathbb{Q}(\sqrt{d}).$ Let $\varepsilon$ be a positive real number. There is a positive number $C,$ which is effectively computable in terms of $\varepsilon$ and $d,$ such that if $k$ exceeds $C$ then
$$
\log p_k<(1+\varepsilon)\log k.
$$
\end{lem}

\begin{proof} Let $K=\mathbb{Q}(\sqrt{d})$ and denote the ring of algebraic integers of $K$ by $\mathcal{O}_K.$ A prime $p$ is the norm of an element $\pi$ of $\mathcal{O}_K$ provided that it is representable as the value of the primitive quadratic form $q_K(x,y)$ given by $x^2-dy^2$ if $d\not\equiv 1\pmod{4}$ and $x^2+xy+\left(\frac{1-d}{4}\right)y^2$ if $d\equiv 1\pmod{4}.$ By \cite[Chapter VII, (2.14)]{FT}, a prime $p$ is represented by $q_K(x,y)$ if and only if $p$ is not inert in $K$ and all prime ideals $\wp$ of $\mathcal{O}_K$ above $p$ have trivial narrow class in the narrow ideal class group of $K.$ Let $K_H$ be the strict Hilbert class field of $K.$ $K_H$ is normal over $K$ and $G,$ the Galois group of $K_H$ over $K,$ is isomorphic with the narrow ideal class group of $K$ and so $|G|=h^+,$ the strict ideal class number of $K,$ see Theorem 7.1.2 of \cite{Co2}. The prime ideals $\wp$ of $\mathcal{O}_K$ which do not ramify in $K_H$ and which are principal are the only prime ideals of $\mathcal{O}_K$ which do not ramify in $K_H$ and which split completely in $K_H,$ see Theorem 7.1.3 of \cite{Co2}. These prime ideals may be counted by the Chebotarev Density Theorem. Let $\left[\frac{K_H/K}{\wp}\right]$ denote the conjugacy class of Frobenius automorphisms corresponding to prime ideals $P$ of $\mathcal{O}_{K_H}$ above $\wp.$ In particular, for each conjugacy class $C$ of $G$ we define $\pi_C(x,K_H/K)$ to be the cardinality of the set of prime ideals $\wp$ of $\mathcal{O}_K$ which are unramified in $K_H,$ for which $\left[\frac{K_H/K}{\wp}\right]=C$ and for which $N_{K/\mathbb{Q}}\wp\leq x.$ Denote by $C_0$ the conjugacy class consisting of the identity element of $G.$ Note that the number of inert primes $p$ of $\mathcal{O}_K$ for which $N_{K/\mathbb{Q}}\ p\leq x$ is at most $x^{1/2}.$ Thus the number of primes $p$ up to $x$ for which $p$ is the norm of an element $\pi$ of $\mathcal{O}_K$ is bounded from below by
\begin{equation} \label{eq13}
\pi_{C_0}(x,K_H/K)-x^{1/2}.
\end{equation}
It follows from Theorem 1.3 and 1.4 of \cite{LO} that there is a positive number $C_1,$ which is effectively computable in terms of $d,$ such that for $x$ greater than $C_1$ the quantity \eqref{eq13} exceeds
$$
\frac{x}{2h^+\log x}.
$$
Further
$$
\frac{x}{2h^+\log x}>k
$$
when $x$ is at least $4h^+ k\log k$ and
\begin{equation} \label{eq14}
k/\log k>4h^+.
\end{equation}
Thus, provided \eqref{eq14} holds and $x$ exceeds $C_1,$
\begin{equation} \label{eq15}
p_k<4h^+k\log k.
\end{equation}
Our result now follows from \eqref{eq15} on taking logarithms.
\end{proof}

\section{Estimates for linear forms in $p$-adic logarithms of algebraic numbers}

Let $\alpha_1,\dots,\alpha_n$ be non-zero algebraic numbers and put $K=\mathbb{Q}(\alpha_1,\dots,\alpha_n)$ and $d=[K:\mathbb{Q}].$ Let $\wp$ be a prime ideal of the ring $\mathcal{O}_K$ of algebraic integers in $K$ lying above the prime number $p.$ Denote by $e_\wp$ the ramification index of $\wp$ and by $f_\wp$ the residue class degree of $\wp.$ For $\alpha$ in $K$ with $\alpha\neq 0$ let ord$_\wp\alpha$ be the exponent to which $\wp$ divides the principal fractional ideal generated by $\alpha$ in $K$ and put ord$_\wp0=\infty.$ For any positive integer $m$ let $\zeta_m=e^{2\pi i/m}$ and put $\alpha_0=\zeta_{2^u}$ where $\zeta_{2^u}\in K$ and $\zeta_{2^{u+1}}\not\in K.$

Suppose that $\alpha_1,\dots,\alpha_n$ are multiplicatively independent $\wp$-adic units in $K.$ Let $\overline{\alpha_0},\overline{\alpha_1},\dots,\overline{\alpha_n}$ be the images of $\alpha_0,\alpha_1,\dots,\alpha_n$ under the residue class map at $\wp$ from the ring of $\wp$-adic integers in $K$ onto the residue class field $\overline{K}_{\wp}$ at $\wp.$ For any set $X$ let $|X|$ denote its cardinality. Let $\langle \overline{\alpha_0},\overline{\alpha_1},\dots,\overline{\alpha_n}\rangle$ be the subgroup of $(\overline{K}_{\wp})^\times$ generated by $\overline{\alpha_0},\overline{\alpha_1},\dots,\overline{\alpha_n}.$ We define $\delta$ by
$$
\delta=1\quad \text{ if }\quad [K(\alpha_0^{1/2},\alpha_1^{1/2},\dots,\alpha_n^{1/2}):K]<2^{n+1}
$$
and
$$
\delta=(p^{f_\wp}-1)/|\langle\overline{\alpha_0},\overline{\alpha_1},\dots,\overline{\alpha_n}\rangle|
$$
if
\begin{equation} \label{eq16}
[K(\alpha_0^{1/2},\alpha_1^{1/2},\dots,\alpha_n^{1/2}):K]=2^{n+1}.
\end{equation}

\begin{lem} \label{lem5}
Let $p$ be a prime with $p\geq 5$ and let $\wp$ be an unramified prime ideal of $\mathcal{O}_K$ lying above $p.$ Let $\alpha_1,\dots,\alpha_n$ be multiplicatively independent $\wp$-adic units. Let $b_1,\dots,b_n$ be integers, not all zero, and put
$$
B=\max(2,|b_1|,\dots,|b_n|).
$$
Then
$$
\text{\rm ord}_\wp(\alpha_1^{b_1}\cdots\alpha_n^{b_n}-1)<Ch(\alpha_1)\cdots h(\alpha_n)\max(\log B,(n+1)(5.4n+\log d))
$$
where
\begin{equation*}
\begin{split}
C & =376(n+1)^{1/2}\left(7e\frac{p-1}{p-2}\right)^nd^{n+2}\log^*d\log(e^4(n+1)d)\cdot \\
& \qquad \max\left(\frac{p^{f_p}}{\delta}\left(\frac{n}{f_p\log p}\right)^n,e^nf_p\log p\right).
\end{split}
\end{equation*}
\end{lem}

\begin{proof} We apply the Main Theorem of \cite{SY2} and in (1.16) we take $C_1(n,d,\wp,a)h^{(1)}$ in place of the minimum. Further (1.15) holds since our result is symmetric in the $b_i$'s. Next we note that, since $\wp$ is unramified and $p\geq 5,$ we may take $c^{(1)}=1794,$ $a^{(1)}=7\frac{p-1}{p-2},$ $a^{(1)}_0=2+\log 7$ and $a^{(1)}_1=a^{(1)}_2=5.25.$ We remark that condition \eqref{eq16} ensures that we may take $\{\theta_1,\dots,\theta_n\}$ to be $\{\alpha_1,\dots,\alpha_n\}.$ Finally the explicit version of Dobrowolski's Theorem due to Voutier \cite{Va3} allows us to replace the first term in the maximum defining $h^{(1)}$ by $\log B.$ Therefore we find that
$$
\text{ord}_\wp(\alpha_1^{b_1}\cdots\alpha_n^{b_n}-1)<C_1h(\alpha_1)\cdots h(\alpha_n)\max(\log B,G_1,(n+1)f_\wp\log p)
$$
where
$$
G_1=(n+1)((2+\log 7)n+5.25+\log((2+\log 7)n+5.25)+\log d)
$$
and, on denoting $\log\max(x,e)$ by $\log^*x,$
\begin{equation*}
\begin{split}
C_1 & =1794\left(7\left(\frac{p-1}{p-2}\right)\right)^n\frac{(n+1)^{n+1}}{n!}\,\frac{d^{n+2}\log^*d}{2^u(f_\wp\log p)^2} \\
&\qquad \cdot\max\left(\frac{p^{f_\wp}}{\delta}\left(\frac{n}{f_\wp\log p}\right)^n,e^nf_\wp\log p\right)\cdot \max(\log(e^4(n+1)d),f_\wp\log p).
\end{split}
\end{equation*}

Note that $2^u\geq 2$ and $f_\wp\log p\geq \log 5.$ Further, by Stirling's formula, see 6.1.38 of \cite{AS},
$$
\frac{(n+1)^{n+1}}{n!}\leq\frac{e^{n+1}(n+1)^{1/2}}{\sqrt{2\pi}}
$$
and so
\begin{equation} \label{eq17}
\text{ord}_\wp(\alpha_1^{b_1}\cdots\alpha_n^{b_n}-1)<C_2 h(\alpha_1)\cdots h(\alpha_n)\max\left(\frac{\log B}{\log 5},\frac{G_1}{\log 5},n+1\right)
\end{equation}
where
\begin{equation} \label{eq18}
\begin{split}
C_2 & =\frac{1794}{2}\,\frac{e}{\sqrt{2\pi}}(n+1)^{1/2}\left(7e\frac{p-1}{p-2}\right)^nd^{n+2}\log^*d \\
& \qquad \max\left(\frac{p^{f_\wp}}{\delta}\left(\frac{n}{f_\wp\log p}\right)^n,e^nf_\wp\log p\right)\frac{(\log(e^4(n+1)d)}{\log 5}.
\end{split}
\end{equation}
We next observe that 
$$
G_1\leq (n+1)(5.4n+\log d)
$$
and as a consequence
\begin{equation} \label{eq19}
\max\left(\frac{\log B}{\log 5},\frac{G_1}{\log 5},n+1\right)=\max\left(\frac{\log B}{\log 5},\frac{(n+1)(5.4n+\log d)}{\log 5}\right).
\end{equation}
The result now follows from \eqref{eq17}, \eqref{eq18} and \eqref{eq19}.
\end{proof}

\section{Further preliminaries}

Let $(\alpha+\beta)^2$ and $\alpha\beta$ be non-zero integers with $\alpha/\beta$ not a root of unity. We may suppose that $|\alpha|\geq|\beta|.$ Since there is a positive number $c_0$ which exceeds 1 such that $|\alpha|\geq c_0$ we deduce from Lemma 3 of \cite{St3}, see also Lemmas 1 and 2 of \cite{Sc6}, that there is a positive number $c_1$ which we may suppose exceeds $(\log c_0)^{-1}$ such that for $n>0$
\begin{equation} \label{eq20}
\log 2+n\log|\alpha|\geq\log|\alpha^n-\beta^n|\geq(n-c_1\log(n+1))\log|\alpha|.
\end{equation}
The proof of \eqref{eq20} depends upon an estimate for a linear form in the logarithms of two algebraic numbers due to Baker \cite{B1}.

For any positive integer $n$ let $\mu(n)$ denote the M\"obius function of $n.$ It follows from \eqref{eq4} that
\begin{equation} \label{eq21}
\Phi_n(\alpha,\beta)=\prod_{d\mid n}(\alpha^{n/d}-\beta^{n/d})^{\mu(d)}.
\end{equation}
We may now deduce, following the approach of \cite{Sc6} and \cite{St3}, our next result.

\begin{lem} \label{lem6}
There exists an effectively computable positive number $c$ such that if $n>2$ then
\begin{equation} \label{eq22}
|\alpha|^{\varphi(n)-cq(n)\log n}\leq|\Phi_n(\alpha,\beta)|\leq|\alpha|^{\varphi(n)+cq(n)\log n},
\end{equation}
where $q(n)=2^{\omega(n)}.$
\end{lem}

\begin{proof}
By \eqref{eq21}
$$
\log|\Phi_n(\alpha,\beta)|=\sum_{d\mid n}\mu(d)\log|\alpha^{n/d}-\beta^{n/d}|
$$
and so by \eqref{eq20}
$$
\left|\log|\Phi_n(\alpha,\beta)|-\sum_{d\mid n}\mu(d)\frac{n}{d}\log|\alpha|\right|\leq\sum_{\substack{d\mid n \\ \mu(d)\neq 0}}c_1\log(n+1)\log|\alpha|
$$
since $c_1$ exceeds $(\log c_0)^{-1}.$ Our result now follows.
\end{proof}

\begin{lem} \label{lem7}
There exists an effectively computable positive number $c_1$ such that if $n$ exceeds $c_1$ then
\begin{equation} \label{eq23}
\log|\Phi_n(\alpha,\beta)|\geq\frac{\varphi(n)}{2}\log|\alpha|.
\end{equation}
\end{lem}

\begin{proof}
For $n$ sufficiently large
$$
\varphi(n)>n/2\log\log n\quad \text{and}\quad q(n)<n^{1/\log\log n}.
$$
Since $|\alpha|\geq c_0>1$ it follows from \eqref{eq22} that if $n$ is sufficiently large
$$
|\Phi_n(\alpha,\beta)|>|\alpha|^{\varphi(n)/2},
$$
as required.
\end{proof}

\begin{lem} \label{lem8}
Let $n$ be an integer larger than $1,$ let $p$ be a prime which does not divide $\alpha\beta$ and let $\wp$ be a prime ideal of the ring of algebraic integers of $\mathbb{Q}(\alpha/\beta)$ lying above $p$ which does not ramify. Then there exists a positive number $C,$ which is effectively computable in terms of $\omega(\alpha\beta)$ and the discriminant of $\mathbb{Q}(\alpha/\beta),$ such that if $p$ exceeds $C$ then
$$
\text{ord}_\wp((\alpha/\beta)^n-1)<p\exp(-\log p/51.9\log\log p)\log|\alpha|\log n.
$$
\end{lem}

\begin{proof}
Let $c_1,c_2,\dots$ denote positive numbers which are effectively computable in terms of $\omega(\alpha\beta)$ and the discriminant of $\mathbb{Q}(\alpha/\beta).$ We remark that since $\alpha/\beta$ is of degree at most 2 over $\mathbb{Q}$ the discriminant of $\mathbb{Q}(\alpha/\beta)$ determines the field $\mathbb{Q}(\alpha/\beta)$ and so knowing it one may compute the class number and regulator of $\mathbb{Q}(\alpha/\beta)$ as well as the strict Hilbert class field of $\mathbb{Q}(\alpha/\beta)$ and the discriminant of this field. Further let $p$ be a prime which does not divide $6d\alpha\beta$ where $d$ is defined as in the first paragraph of \S2.

Put $K=\mathbb{Q}(\alpha/\beta)$ and
$$
\alpha_0=\begin{cases}
i & \text{if}\ i\in K \\
-1 & \text{otherwise}.
\end{cases}
$$
Let $v$ be the largest integer for which
\begin{equation} \label{eq24}
\alpha/\beta=\alpha^j_0\theta^{2^v},
\end{equation}
with $0\leq j\leq 3$ and $\theta$ in $K.$ To see that there is a largest such integer we first note that either there is a prime ideal $\mathfrak{q}$ of $\mathcal{O}_K,$ the ring of algebraic integers of $K,$ lying above a prime $q$ which occurs to a positive exponent in the principal fractional ideal generated by $\alpha/\beta$ or $\alpha/\beta$ is a unit. In the former case $h(\alpha/\beta)\geq 2^{v-1}\log q$ and in the latter case, since $\alpha/\beta$ is not a root of unity, there is a positive number $c_1,$ see \cite{Do}, such that $h(\alpha/\beta)\geq 2^vc_1.$

Notice from \eqref{eq24} that
\begin{equation} \label{eq25}
h(\alpha/\beta)=2^vh(\theta).
\end{equation}
Further, by Kummer theory, see Lemma 3 of \cite{BS},
\begin{equation} \label{eq26}
[K(\alpha^{1/2}_0,\theta^{1/2}):K]=4.
\end{equation}
Furthermore since $p\nmid\alpha\beta$ and $\alpha$ and $\beta$ are algebraic integers
\begin{equation} \label{eq27}
\text{ord}_\wp((\alpha/\beta)^n-1)\leq\text{ord}_\wp((\alpha/\beta)^{4n}-1).
\end{equation}

For any real number $x$ let $[x]$ denote the greatest integer less than or equal to $x.$ Put
\begin{equation} \label{eq28}
k=\left[\frac{\log p}{51.8\log\log p}\right].
\end{equation}
Then, for $p>c_2,$ we find that $k\geq 2$ and
\begin{equation} \label{eq29}
\max\left(p\left(\frac{k}{\log p}\right)^k,e^k\log p\right)=p\left(\frac{k}{\log p}\right)^k.
\end{equation}

Our proof splits into two cases. We shall first suppose that $\mathbb{Q}(\alpha/\beta)=\mathbb{Q}$ so that $\alpha$ and $\beta$ are integers. For any positive integer $j$ with $j\geq 2$ let $p_j$ denote the $j-1$-th smallest prime which does not divide $p\alpha\beta.$ We put
\begin{equation} \label{eq30}
m=n2^{v+2}
\end{equation}
and
$$
\alpha_1=\theta/p_2\cdots p_k.
$$
Then
\begin{equation} \label{eq31}
\theta^m-1=\left(\frac{\theta}{p_2\cdots p_k}\right)^mp^m_2\cdots p^m_k-1=\alpha^m_1p^m_2\cdots p^m_k-1
\end{equation}
and by \eqref{eq24}, \eqref{eq27}, \eqref{eq30} and \eqref{eq31}
\begin{equation} \label{eq32}
\text{ord}_p((\alpha/\beta)^n-1)\leq\text{ord}_p(\alpha^m_1p^m_2\cdots p^m_k-1).
\end{equation}

Note that $\alpha_1,p_2,\dots,p_k$ are multiplicatively independent since $\alpha/\beta$ is not a root of unity and $p_2,\dots,p_k$ are primes which do not divide $p\alpha\beta.$ Further, since $p_2,\dots,p_k$ are different from $p$ and $p$ does not divide $\alpha\beta,$ we see that $\alpha_1,p_2,\dots,p_k$ are $p$-adic units.

We now apply Lemma 5 with $\delta=1$, $d=1,$ $f_\wp=1$ and $n=k$ to conclude that
\begin{equation} \label{eq33}
\begin{split}
\text{ord}_p(\alpha^m_1p^m_2\cdots p^m_k-1) & \leq c_3(k+1)^3\log p\left(7e\frac{p-1}{p-2}\right)^k \\
& \qquad \max\left(p\left(\frac{k}{\log p}\right)^k,e^k\right)(\log m)h(\alpha_1)\log p_2\cdots \log p_k.
\end{split}
\end{equation}

Put
$$
t=\omega(\alpha\beta).
$$
Let $q_i$ denote the $i$-th prime number. Note that
$$
p_k\leq q_{k+t+1}
$$
and thus
$$
\log p_2+\cdots+\log p_k\leq (k-1)\log q_{k+t+1}.
$$
By the prime number theorem with error term, for $k>c_4,$
\begin{equation} \label{eq34}
\log p_2+\cdots+\log p_k\leq 1.001(k-1)\log k.
\end{equation}
By the arithmetic geometric mean inequality
$$
\log p_2\cdots\log p_k\leq\left(\frac{\log p_2+\cdots+\log p_k}{k-1}\right)^{k-1}
$$
and so, by \eqref{eq34},
\begin{equation} \label{eq35}
\log p_2\cdots \log p_k\leq (1.001 \log k)^{k-1}.
\end{equation}

Since $h(\alpha_1)\leq h(\theta)+\log p_2\cdots p_k$ it follows from \eqref{eq34} that
\begin{equation} \label{eq36}
h(\alpha_1)\leq c_5h(\theta)k\log k.
\end{equation}
Further $m=2^{v+2}n$ is at most $n^{2^{v+2}}$ and so by \eqref{eq9} and \eqref{eq25}
\begin{equation} \label{eq37}
h(\theta)\log m\leq 4h(\alpha/\beta)\log n\leq 4\log |\alpha|\log n.
\end{equation}
Thus, by \eqref{eq29}, \eqref{eq32}, \eqref{eq33}, \eqref{eq35}, \eqref{eq36} and \eqref{eq37},
$$
\text{ord}_p((\alpha/\beta)^n-1)<c_6k^4\log p\left(7e\frac{p-1}{p-2}1.001\frac{k\log k}{\log p}\right)^kp\log|\alpha|\log n.
$$
Therefore, by \eqref{eq28}, for $p>c_7$
\begin{equation} \label{eq38}
\text{ord}_p((\alpha/\beta)^n-1)<pe^{-\frac{\log p}{51.9\log\log p}}\log|\alpha|\log n.
\end{equation}

We now suppose that $[\mathbb{Q}(\alpha/\beta):\mathbb{Q}]=2.$ Let $\pi_2,\dots,\pi_k$ be elements of $\mathcal{O}_K$ with the property that $N(\pi_i)=p_i$ where $N$ denotes the norm from $K$ to $\mathbb{Q}$ and where $p_i$ is a rational prime number which does not divide $2dp\alpha\beta.$ We now put $\theta_i=\pi_i/\pi'_i$ where $\pi'_i$ denotes the algebraic conjugate of $\pi_i$ in $\mathbb{Q}(\alpha/\beta).$ Notice that $p$ does not divide $\pi_i\pi'_i=p_i$ and if $p$ does not divide $(\pi_i-\pi'_i)^2$ then
$$
\left(\frac{(\pi_i-\pi'_i)^2}{p}\right)=\left(\frac{d}{p}\right),
$$
since $\mathbb{Q}(\alpha/\beta)=\mathbb{Q}(\sqrt{d})=\mathbb{Q}(\pi_i).$ Thus, by Lemma 2, the order of $\theta_i$ in $(\overline{\mathbb{Q}(\alpha/\beta})_\wp)^\times$ is a divisor of 2 if $p$ divides $\pi^2_i-\pi^{\prime 2}_i$ and a divisor of $p-\left(\frac{d}{p}\right)$ otherwise. Since $p$ is odd and $p$ does not divide $d$ we conclude that the order of $\theta_i$ in $(\overline{\mathbb{Q}(\alpha/\beta})_\wp)^\times$ is a divisor of $p-\left(\frac{d}{p}\right).$

Put
\begin{equation} \label{eq39}
\alpha_1=\theta/\theta_2\cdots\theta_k.
\end{equation}
Then
$$
\theta^m-1=\left(\frac{\theta}{\theta_2\cdots\theta_k}\right)^m\theta^m_2\cdots\theta^m_k-1
$$
and, by \eqref{eq24}, \eqref{eq27}, \eqref{eq30} and \eqref{eq39},
\begin{equation} \label{eq40}
\text{ord}_\wp((\alpha/\beta)^n-1)\leq\text{ord}_\wp(\alpha^m_1\theta^m_2\cdots\theta^m_k-1).
\end{equation}

Observe that $\alpha_1,\theta_2,\dots,\theta_k$ are multiplicatively independent since $\alpha/\beta$ is not a root of unity, $p_2,\dots,p_k$ are primes which do not divide $\alpha\beta$ and the principal prime ideals $[\pi_i]$ for $i=2,\dots,k$ do not ramify since $p\nmid 2d.$ Further since $p_2,\dots,p_k$ are different from $p$ and $p$ does not divide $\alpha\beta$ we see that $\alpha_1,\theta_2,\dots,\theta_k$ are $\wp$-adic units.

Notice that
$$
K(\alpha^{1/2}_0,\theta^{1/2},\theta^{1/2}_2,\dots,\theta^{1/2}_k)=K(\alpha^{1/2}_0,\alpha^{1/2}_1,\theta^{1/2}_2,\dots,\theta^{1/2}_k).
$$
Further
\begin{equation} \label{eq41}
[K(\alpha^{1/2}_0,\theta^{1/2},\theta^{1/2}_2,\dots,\theta^{1/2}_k):K]=2^{k+1},
\end{equation}
since otherwise, by \eqref{eq26} and Kummer theory, see Lemma 3 of \cite{BS}, there is an integer $i$ with $2\leq i\leq k$ and integers $j_0,\dots,j_{i-1}$ with $0\leq j_b\leq 1$ for $b=0,\dots,i-1$ and an element $\gamma$ of $K$ for which
\begin{equation} \label{eq42}
\theta_i=\alpha^{j_0}_0\theta^{j_1}\theta^{j_2}_2\cdots\theta^{j_{i-1}}_{i-1}\gamma^2.
\end{equation} 
But the order of the prime ideal $[\pi_i]$ on the left-hand side of \eqref{eq42} is 1 whereas the order on the right-hand side of \eqref{eq42} is even which is a contradiction. Thus \eqref{eq41} holds.

Since $p$ does not divide the discriminant of $K$ and $[K:\mathbb{Q}]=2$ either $p$ splits, in which case $f_\wp=1$ and $\left(\frac{d}{p}\right)=1$, or $p$ is inert, in which case $f_\wp=2$ and $\left(\frac{d}{p}\right)=-1,$ see \cite{H}. Observe that if $\left(\frac{d}{p}\right)=1$ then
\begin{equation} \label{eq43}
p^{f_p}/\delta\leq p.
\end{equation}

Let us now determine $|\langle \overline{\alpha_0},\overline{\theta},\overline{\theta_2},\dots,\overline{\theta_k}\rangle|$ in the case $\left(\frac{d}{p}\right)=-1.$ By our earlier remarks the order of $\overline{\theta_i}$ is a divisor of $p+1$ for $i=2,\dots,k.$ Further by \eqref{eq24} since $N(\alpha/\beta)=1$ we find that $N(\theta)=\pm 1$ and so $N(\theta^2)=1.$ By Hilbert's Theorem 90, see eg.\ Theorem 14.35 of \cite{Co1}, $\theta^2=\varrho/\varrho'$ where $\varrho$ and $\varrho'$ are conjugate algebraic integers in $\mathbb{Q}(\alpha/\beta).$ Note that we may suppose that the principal ideals $[\varrho]$ and $[\varrho']$ have no principal ideal divisors in common. Further since $p$ does not divide $\alpha\beta$ and, since $\left(\frac{d}{p}\right)=-1,$ $[p]$ is a principal prime ideal of $\mathcal{O}_K$ and we note that $p$ does not divide $\varrho\varrho'.$ It follows from Lemma 2 that the order of $\theta^2$ in $(\overline{\mathbb{Q}(\alpha/\beta})_\wp)^\times$ is a divisor of $p+1$ hence $\theta$ has order a divisor of $2(p+1).$ Since $\alpha^4_0=1$ we conclude that 
$$
|\langle \overline{\alpha_0},\overline{\theta},\overline{\theta_2},\dots,\overline{\alpha_k}\rangle |\leq 2(p+1)
$$ 
and so
\begin{equation} \label{eq44}
\delta=(p^2-1)/|\langle \overline{\alpha_0},\overline{\theta},\overline{\theta_2},\dots,\overline{\theta_k}\rangle|\geq (p-1)/2.
\end{equation}

We now apply Lemma 5 noting, by \eqref{eq43} and \eqref{eq44}, that
$$
p^{f_p}/\delta\leq 2p^2/(p-1).
$$
Thus, by \eqref{eq29},
\begin{equation} \label{eq45}
\begin{split}
\text{ord}_\wp(\alpha^m_1\theta^m_2\cdots\theta^m_k-1) & \leq c_8(k+1)^3\log p\left(7e\frac{p-1}{p-2}\right)^k2^kp\left(\frac{k}{\log p}\right)^k\\ 
& \qquad (\log m) h(\alpha_1)h(\theta_2)\cdots h(\theta_k).
\end{split}
\end{equation}

Notice that $\theta_i=\pi_i/\pi'_i$ and that $p_i(x-\pi_i/\pi'_i)(x-\pi'_i/\pi_i)=p_ix^2-(\pi^2_i+\pi'^{2}_i)x+p_i$ is the minimal polynomial of $\theta_i$ over the integers since $[\pi_i]$ is unramified. Either the discriminant of $\mathbb{Q}(\alpha/\beta)$ is negative in which case $|\pi_i|=|\pi'_i|$ or it is positive in which case there is a fundamental unit $\varepsilon>1$ in $\mathcal{O}_K.$ We may replace $\pi_i$ by $\pi_i\varepsilon^u$ for any integer $u$ and so without loss of generality we may suppose that $p^{1/2}_i\leq|\pi_i|\leq p^{1/2}_i\varepsilon$ and consequently that $p_i^{1/2}\varepsilon^{-1}\leq|\pi'_i|\leq p_i^{1/2}.$ Therefore
$$
h(\theta_i)\leq\frac{1}{2}\log p_i\varepsilon^2=\frac{1}{2}\log p_i+\log\varepsilon\quad \text{for}\ d>0
$$
and
$$
h(\theta_i)\leq\frac{1}{2}\log p_i\quad \text{for}\ d<0.
$$
Let us put
$$
R=\begin{cases}
\log\varepsilon & \text{for}\ d>0 \\
0 & \text{for}\ d<0.
\end{cases}
$$
Then
\begin{equation} \label{eq46}
h(\theta_i)\leq\frac{1}{2}\log p_i+R
\end{equation}
for $i=2,\dots,k.$ In a similar fashion we find that
\begin{equation} \label{eq47}
h(\theta_2\cdots\theta_k)\leq\frac{1}{2}\log p_2\cdots p_k+R,
\end{equation}
and so
\begin{equation} \label{eq48}
h(\alpha_1)\leq h(\theta)+\frac{1}{2}\log p_2\cdots p_k+R.
\end{equation}

Put
$$
t_1=\omega(2dp\alpha\beta).
$$
Let $q_i$ denote the $i$-th prime number which is representable as the norm of an element of $\mathcal{O}_K.$ Note that
$$
p_k\leq q_{k+t_1}
$$
and thus
$$
\log p_2+\cdots+\log p_k\leq (k-1)\log q_{k+t_1}.
$$
Therefore by Lemma 4 for $k>c_9,$
\begin{equation} \label{eq49}
\log p_2+\cdots+\log p_k\leq 1.0005(k-1)\log k
\end{equation}
and so, as for the proof of \eqref{eq35},
$$
\log p_2\cdots\log p_k\leq(1.0005\log k)^{k-1}.
$$
Accordingly, since $p_k\geq k,$ for $k>c_{10}$
\begin{equation} \label{eq50}
2^{k-1}h(\theta_2)\cdots h(\theta_k)\leq (\log p_2+2R)\cdots (\log p_k+2R)\leq(1.001\log k)^{k-1}.
\end{equation}
Furthermore as for the proof of \eqref{eq36} and \eqref{eq37} we find that from \eqref{eq48},
\begin{equation} \label{eq51}
h(\alpha_1)\leq c_{11}h(\theta)k\log k
\end{equation}
and, from \eqref{eq9}, \eqref{eq25} and \eqref{eq30},
\begin{equation} \label{eq52}
h(\theta)\log m\leq 8\log |\alpha|\log n.
\end{equation}

Thus by  \eqref{eq40}, \eqref{eq45}, \eqref{eq48}, \eqref{eq50}, \eqref{eq51} and \eqref{eq52},
\begin{equation} \label{eq53}
\text{ord}_\wp((\alpha/\beta)^n-1)<c_{12}k^4\log p\left(7e\left(\frac{p-1}{p-2}\right)1.001\frac{k\log k}{\log p}\right)^k p\log|\alpha|\log n.
\end{equation}
Therefore, by \eqref{eq28}, for $p>c_{13}$ we again obtain \eqref{eq38} and the result follows.
\end{proof}

\section{Proof of Theorem 1}

Let $c_1,c_2,\dots$ denote positive numbers which are effectively computable in terms of $\omega(\alpha\beta)$ and the discriminant of $\mathbb{Q}(\alpha/\beta).$ Let $g$ be the greatest common divisor of $(\alpha+\beta)^2$ and $\alpha\beta.$ Note that $\varphi(n)$ is even for $n>2$ and that
$$
\Phi_n(\alpha,\beta)=g^{\varphi(n)/2}\Phi_n(\alpha_1,\beta_1)
$$
where $\alpha_1=\alpha/\sqrt{g}$ and $\beta_1=\beta/\sqrt{g}.$ Further $(\alpha_1+\beta_1)^2$ and $\alpha_1\beta_1$ are coprime and plainly
$$
P(\Phi_n(\alpha,\beta))\geq P(\Phi_n(\alpha_1,\beta_1)).
$$
Therefore we may assume, without loss of generality, that $(\alpha+\beta)^2$ and $\alpha\beta$ are coprime non-zero integers. 

By Lemma 7 there exists $c_1$ such that if $n$ exceeds $c_1$ then
\begin{equation} \label{eq54}
\log|\Phi_n(\alpha,\beta)|\geq\frac{\varphi(n)}{2}\log|\alpha|.
\end{equation}

On the other hand
\begin{equation} \label{eq55}
\Phi_n(\alpha,\beta)=\prod_{p\mid\Phi_n(\alpha,\beta)}p^{\text{ord}_p\Phi_n(\alpha,\beta)} \qquad \qquad.
\end{equation}
If $p$ divides $\Phi_n(\alpha, \beta)$ then, by \eqref{eq4}, $p$ does not divide $\alpha\beta$ and so
\begin{equation} \label{eq56}
\text{ord}_p\Phi_n(\alpha,\beta)\leq\text{ord}_\wp((\alpha/\beta)^n-1), 
\end{equation}
where $\wp$ is a prime ideal of $\mathcal{O}_K$  lying above $p.$ By Lemma 1 if $p$ divides $\Phi_n(\alpha,\beta)$ and $p$ is not $P(n/(3,n))$ then $p$ is at least $n-1$ and thus for $n>c_2$, by Lemma 8,
\begin{equation} \label{eq57}
\text{ord}_\wp((\alpha/\beta)^n-1)<p\exp(-\log p/51.9 \log\log p)\log|\alpha|\log n.
\end{equation}

Put
$$
P_n=P(\Phi_n(\alpha,\beta)).
$$
Then, by \eqref{eq55} and Lemma 1,
\begin{equation} \label{eq58}
\log|\Phi_n(\alpha,\beta)|\leq\log n+\sum_{\substack{p\leq P_n \\ p\nmid n}}\log p\ \text{ord}_p\Phi_n(\alpha,\beta).
\end{equation}

Comparing \eqref{eq54} and \eqref{eq58} and using \eqref{eq56} and \eqref{eq57} we find that, for $n>c_3,$
$$
\varphi(n)\log|\alpha|<\sum_{\substack{p\leq P_n \\ p\nmid n}} c_4(\log p)p\exp(-\log p/51.9\log\log p)\log |\alpha|\log n
$$
hence
$$
\frac{\varphi(n)}{\log n}<(\pi(P_n,n,1)+\pi(P_n,n,-1))P_n\exp(-\log P_n/51.95\log\log P_n),
$$
and, by Lemma 3,
$$
c_5\frac{\varphi(n)}{\log n}<\frac{P^2_n}{\varphi(n)\log(P_n/n)}\exp(-\log P_n/51.95\log\log P_n).
$$
Since $\varphi(n)>c_6n/\log\log n,$
$$
P_n>n\exp(\log n/104\log\log n),
$$
for $n>c_7,$ as required.\\

\section{Proof of Theorem 2}

Since $p$ does not divide $a b$ 

$$\text{ord}_p ~ (a^n - b^n) = \text{ord}_p ((a/g)^n - (b/g)^n)$$ 

\noindent
where $g$ is the greatest common divisor of $a$ and $b$.  Thus we may assume, without loss of generality, that $a$ and $b$ are coprime. Put $u_n=a^n-b^n$ for $n=1,2,\cdots $ and   let $\ell=\ell(p)$ be the smallest positive integer for which $p$ divides $u_\ell$.  Certainly $\ell$ divides $p-1$.  Further, as in the proof of Lemma 3 of \cite{St2}, if $n$ and $m$ are positive integers then

$$(u_n, u_m) = u_{(n,m)}.$$

\noindent
Thus if $p$ divides $u_n$ then $p$ divides $u_{(n,\ell)}$.  By the minimality of $\ell$ we see that $(n,\ell)=\ell$ so that $\ell$ divides $n$.  In particular $\ell$ divides $p-1$.  Furthermore, by $(4)$, we see that

$$\text{ord}_p u_\ell = \text{ord}_p \Phi_\ell (a, b).$$

\noindent
If $\ell$ divides $n$ then, by Lemma 2 of \cite{St2},

\begin{equation} \label{eq62}
(u_n/u_{\ell}, u_{\ell}) ~ \text{divides} ~ n/{\ell},\\
\end{equation}

\noindent
 and so

\begin{equation} \label{eq63}
\text{ord}_p ~ u_{p-1} = \text{ord}_p u_{\ell}
\end{equation}\\

\noindent
Suppose that $p$ divides $\Phi_n(a,b)$.  Then $p$ divides $u_n$ and so $\ell$ divides $n$.  Put $n=t \ell p^k$ with $(t,p)=1$ and $k$ a non-negative integer.  Since $\Phi_n(a, b)$ divides $u_n/u_{n/t} ~\text{for} ~ t > 1$ we see from (62), since $(t,p)=1$, that $t=1$.  Thus $n=\ell p{^k}$.  For any positive integer $m$

$$u_{mp}/u_m = p b^{(m-1)p} + \binom{p}{2}b^{(m-2)p} u_m + \cdots + u_m^{p-1}$$ 

\noindent
and if p is not 2 and p divides $u_m$ then $\text{ord}_p(u_{mp}/u_m)=1$.  It then follows that if $p$ is an odd prime then 
$$\text{ord}_p \Phi_{\ell p ^k} (a, b) = 1 \qquad    \text{for} ~ k = 1, 2, \cdots .$$

If $n$ is a positive integer not divisible by $\ell = \ell(p)$ then $|u_{n}|_p=1$.  On the other hand if $\ell$ divides $n$ and $p$ is odd then 

\begin{equation}\label{eq64}
|u_n|_p=~|u_\ell|_p \cdot | n/\ell |_p  .
\end{equation}

\noindent
It now follows from (63) and (64) and the fact that $\ell \leq p-1$ that if $p$ is an odd prime and $\ell$ divides $n$ then

\begin{equation}\label{eq65}
|u_n|_p= ~~ |u_{p-1} |_p \cdot | n |_p.
\end{equation}

Therefore if $p$ is an odd prime and $n$ a positive integer 

\begin{equation}\label{eq66}
\text{ord}_p(a^n - b ^n) \leq \text{ord}_p(a^{p-1}-b^{p-1})  + \text{ord}_p n, 
\end{equation}

and our result now follows from (66) on taking $n=p-1$ in Lemma 8.

\end{document}